\documentclass[11pt,a4paper]{article}
\usepackage{epsf,epsfig,amsfonts,amsgen,amsmath,amstext,amsbsy,amsopn,amsthm}
\usepackage{amsmath}
\usepackage{enumerate}
\usepackage{hhline}
\usepackage{multirow}
\usepackage{multicol}
\usepackage{booktabs}
\usepackage{lineno}
\usepackage{amsfonts,amsthm,amssymb,bm}
\usepackage{amsfonts}
\usepackage{enumitem}
\usepackage{graphics}
\usepackage{latexsym,bm}
\usepackage{amsfonts,amsthm,amssymb,bbding}
\usepackage{indentfirst}
\usepackage{graphicx}
\usepackage{authblk}
\usepackage{color}
\usepackage[colorlinks=true,anchorcolor=blue,filecolor=blue,linkcolor=red,urlcolor=blue,citecolor=blue]{hyperref}
\usepackage{float}
\usepackage{subcaption}
\usepackage{tikz,enumerate}
\usetikzlibrary{calc}
\usepackage{geometry}
\newcommand{\n}{\noindent}

\newtheorem{theorem}{Theorem}[section]

\newtheorem{lemma}[theorem]{Lemma}
\newtheorem{claim}{Claim}%[section]

\newtheorem{problem}[theorem]{Problem}

\newtheorem{conjecture}[theorem]{Conjecture}

\newtheorem{remark}[theorem]{Remark}

\renewcommand{\leq}{\leqslant}
\renewcommand{\le}{\leqslant}
\renewcommand{\geq}{\geqslant}
\renewcommand{\ge}{\geqslant}

\usepackage[numbers,sort&compress]{natbib}%使引用格式[1,2,3,4,5]变成[1-5]
\begin{document}

\title{\bf An improved polynomial $\chi$-bound for $\{P_5,C_5\}$-free graphs}

%\author[1]{Hongzhang Chen\thanks{Email: mnhzchern@gmail.com.}}
\author[1]{Kaiyang Lan\thanks{Email: kylan95@126.com.}}
\author[1]{Wenlong Zhong\footnote{Email: 2364810512@qq.com.}}

%\author[2]{Yan Wang\thanks{Email: yan.w@sjtu.edu.cn (corresponding author).}}
%\author[3]{Qi Wu\thanks{Email: wuqimath@163.com}}
%
%\affil[1]{School of Mathematics and Statistics, Gansu Center for Applied Mathematics, Lanzhou University, Lanzhou, 730000, Gansu, China}
\affil[1]{School of Mathematics and Statistics, Minnan Normal University, Zhangzhou, 363000, Fujian, China}

%\affil[3]{School of Mathematical Sciences, East China Normal University, Shanghai, 200241, China}

\date{\today}
\maketitle

% \linenumbers \pagewiselinenumbers

% \modulolinenumbers[2]

\begin{abstract}
	Nguyen~\cite{Nguyen2025} recently proved that every $\{P_5,C_5\}$-free graph $G$ satisfies $\chi(G)\leq \omega(G)^{40}$. 
	Building on his framework, we introduce two refinements, namely a sharper cutset decomposition using the $C_5$-free condition and an improved density-increment argument. 
	These yield a polynomial $\chi$-binding function with exponent $24$, improving the previous bound of $40$.
\end{abstract}

{\bf Keywords:} $\{P_5, C_5\}$-free graph, chromatic number, clique number 

{\bf 2020 AMS Subject Classifications:} 05C15, 05C75

\section{Introduction}

All graphs in this paper are finite, simple, and undirected.
For a positive integer $k$, a \textit{$k$-coloring} of a graph $G$ is a mapping $V(G)\to [k]$ such that adjacent vertices receive distinct colors, where $[k]:=\{1,\ldots,k\}$.
The \textit{chromatic number} $\chi(G)$ is the least $k$ for which such a coloring exists.
A set of pairwise adjacent vertices is a \textit{clique}.
The \textit{clique number} $\omega(G)$ is the maximum size of a clique in $G$.
A graph is \textit{perfect} if every induced subgraph $H$ satisfies $\chi(H)=\omega(H)$.
%For a set $X\subseteq V(G)$, $G[X]$ denotes the subgraph induced by $X$, and we write $\chi(X)$ and $\omega(X)$ for $\chi(G[X])$ and $\omega(G[X])$, respectively.
Given a set $\mathcal{H}$ of graphs, we say that $G$ is \textit{$\mathcal{H}$-free} if $G$ has no induced subgraph isomorphic to any graph in $\mathcal{H}$.
If $\mathcal{H}$ consists of a single graph $H$, we write $H$-free for $\{H\}$-free.

%A \textit{hereditary} graph class (also called a monoton e graph class or hereditary property) is a set of graphs that is closed under taking induced subgraphs.

A class $\mathcal{G}$ of graphs is said to be \textit{$\chi$-bounded}~\cite{Gyarfas1973} if there exists a function $f:\mathbb N\to\mathbb N$ such that every induced subgraph $H$ of any $G\in\mathcal{G}$ satisfies $\chi(H)\le f(\omega(H))$. Such an $f$ is called a \textit{$\chi$-binding function} for $\mathcal{G}$.
When $f$ can be chosen to be a polynomial, $\mathcal{G}$ is \textit{polynomially $\chi$-bounded}. Perfect graphs are polynomially $\chi$-bounded with binding function $f(x)=x$.
While $\chi(G)\ge \omega(G)$ always holds, the reverse inequality is generally false; indeed, triangle-free graphs with arbitrarily large chromatic number exist~\cite{Erdos1959,Mycielski1955,Zykov1949}.
A standard necessary condition for $\chi$-boundedness of $H$-free graphs is that $H$ be a forest. The Gy\'arf\'as--Sumner conjecture~\cite{Gyarfas1973,Sumner1981} asserts that this condition is also sufficient:

\begin{conjecture}[\cite{Gyarfas1973,Sumner1981}]\label{conjgs}
	For every forest $T$, the class of $T$-free graphs is $\chi$-bounded.
\end{conjecture}

We use $P_t$ and $C_t$ to denote, respectively, the path and the cycle on $t$ vertices.

Gy\'arf\'as~\cite{Gyarfas1987} proved Conjecture~\ref{conjgs} for $T=P_t$, showing that every $P_t$-free graph $G$ with $t\ge 4$ and $\omega(G)\ge 2$ satisfies $\chi(G)\le (t-1)^{\omega(G)-1}$; this was later improved by Gravier, Ho\`ang, and Maffray~\cite{GravierHoangMaffray2003} to $(t-2)^{\omega(G)-1}$.
Both bounds are exponential in $\omega(G)$.
Esperet~\cite{Esperet2017} conjectured that every $\chi$-bounded class is polynomially $\chi$-bounded, but this was recently disproved by Bria\'nski, Davies, and Walczak~\cite{BrianskiDaviesWalczak2024}, who constructed $\chi$-bounded classes with no polynomial binding function.
Whether $P_t$-free graphs are polynomially $\chi$-bounded remains widely open.

For $t\le 4$, the answer is trivial since $P_4$-free graphs are perfect. 
For $t=5$, polynomial bounds have been established in a series of results, beginning with Sumner~\cite{Sumner1981}, who proved $\chi(G)\le 3$ when $\omega(G)\le 2$.
Esperet, Lemoine, Maffray, and Morel~\cite{EsperetLemoineMaffrayMorel2013} proved $\chi(G)\le 5\cdot 3^{\omega(G)-3}$ for $\omega(G)\ge 3$, and this bound is tight when $\omega(G)=3$.
Scott, Seymour, and Spirkl~\cite{ScottSeymourSpirkl2023} obtained the near-polynomial bound $\chi(G)\le \omega(G)^{\log_2\omega(G)}$ for $\omega(G)\ge 3$.
Very recently, Nguyen~\cite{Nguyen2025} proved the existence of a large absolute constant $d$ such that every $P_5$-free graph satisfies $\chi(G)\le \omega(G)^d$. 
In particular, Nguyen~\cite{Nguyen2025} established the following polynomial bound for $\{P_5,C_5\}$-free graphs:

\begin{theorem}[\cite{Nguyen2025}]\label{thm:P5C540}
	Every $\{P_5,C_5\}$-free graph $G$ satisfies $\chi(G)\le \omega(G)^{40}$.
\end{theorem}

For $t\ge 6$, polynomial $\chi$-boundedness remains largely unexplored; the fact that the best known bounds are still exponential highlights the difficulty of the general case and motivates the study of restricted subclasses, such as those with additional forbidden induced subgraphs; see, for example, \cite{BonomoChudnovskyMaceliSchaudtSteinZhong2018,CameronHuangMerkel2021,CameronHuangPenevSivaraman2020,ChenWuXu2024,ChenXu2025a,ChudnovskyStacho2018,ChudnovskyMaceliStachoZhong2017,ChudnovskyHuangSpirklZhong2021,ChudnovskyKarthickMaceliMaffray2020,ChudnovskySpirklZhong2024a,ChudnovskySpirklZhong2024b,ChoudumKarthickShalu2007,GaspersHuang2019,Huang2024,HuangKarthick2021,KarthickMaffray2019,JuJookenGoedgebeurHuang2026,SchiermeyerRanderath2019} and the references therein. 

In this paper, building on the framework of Nguyen~\cite{Nguyen2025}, we improve Theorem~\ref{thm:P5C540} by showing:

\begin{theorem}\label{thm:intro-main}
	Every $\{P_5,C_5\}$-free graph $G$ satisfies $\chi(G)\le \omega(G)^{24}$.
\end{theorem}
%{\red \bf To here}

%Before sketching our proof method, we introduce some notation and terminology.

%For a vertex \(x \in V(G)\) and a set \(S \subseteq V(G)\), let \(N_S(x)\) denote the set of neighbors of \(x\) in \(S\), i.e., \(N_S(x)=N(x)\cap S\).
%For \(v \in V(G)\), \(N_G(v)\) denotes the neighborhood of \(v\), and \(d_G(v)=|N_G(v)|\) is the degree of \(v\) in \(G\). Let \(\delta(G) = \min_{v \in V(G)} d_G(v)\) denote the minimum degree of \(G\). For a set \(T \subseteq V(G)\), let \(N(T) = \bigcup_{t \in T} N_G(t)\) denote the open neighborhood of \(T\); that is, the set of all vertices outside \(T\) adjacent to at least one vertex of \(T\). (Note that this definition excludes vertices of \(T\) itself, even if they have neighbors in \(T\).) When the graph \(G\) is clear from the context, we omit the subscript.
%Let $H_1$ and $H_2$ be two vertex-disjoint graphs.
%The \textit{join} of $H_1$ and $H_2$, denoted by $H_1 + H_2$, is the graph with vertex set $V(H_1 + H_2) = V(H_1) \cup V(H_2)$ and edge set \(E(H_1 + H_2) = E(H_1) \cup E(H_2) \cup \{ uv : u \in V(H_1),\; v \in V(H_2) \}\).

\subsection*{Outline of the method}

The proof proceeds by refining the cutset decomposition of Nguyen~\cite{Nguyen2025}.
The $C_5$-free condition allows us to reduce the covering multiplicity in a key step from two to one, saving one power of the density parameter.
This is combined with a sharper density-increment argument, which improves the scale change
$y\mapsto y^2$ to $y\mapsto y^{6/5}$.
The two refinements together yield a structural theorem (Theorem~\ref{thm:structure24}) that supports induction on $\omega(G)$ and closes the bound with exponent $24$.

The rest of the paper is organized as follows.
Section~\ref{secpre} collects two preliminary tools for $P_5$-free graphs and records the numerical constants used throughout the paper.
Section~\ref{sec:decomp} develops the refined cutset decomposition and its iteration, culminating in Lemma~\ref{lem:outer-decomp}.
Section~\ref{sec:density} contains the density-increment argument and proves Lemma~\ref{lem:dense-blockade}.
Section~\ref{sec:main-proof} establishes the structural theorem (Theorem~\ref{thm:structure24}) and then applies it to prove Theorem~\ref{thm:intro-main}.
Section~\ref{conremark} gives some concluding remarks.

\section{Preliminaries}\label{secpre}

We mainly follow the notation and terminology of Nguyen~\cite{Nguyen2025}.
Let $G$ be a graph, and let $X$ and $Y$ be two disjoint nonempty subsets of $V(G)$. 
We say that $X$ is \textit{complete} to $Y$ (or the pair $(X,Y)$ is \textit{complete}) if all edges with an end in $X$ and an end in $Y$ are present in $G$, and $X$ is \textit{anticomplete} to $Y$ (or the pair $(X,Y)$ is \textit{anticomplete}) if there are no edges between $X$ and $Y$. 
The pair $(X,Y)$ is \textit{pure} if it is complete or anticomplete.
For $x \in V(G) \setminus Y$, we also say $x$ is \textit{complete} (\textit{anticomplete}) to $Y$ if
$\{x\}$ is \textit{complete} (\textit{anticomplete}) to $Y$.
Similarly, we say that $x$ is \textit{pure} to $Y$ if $(\{x\},Y)$ is pure.
A vertex $x \in V(G) \setminus Y$ is \textit{mixed} to $Y$ if it has both a neighbor and a nonneighbor in $Y$.
%Let $G[X]$ denote the subgraph of $G$ induced by $X$, and $G\setminus X$ the subgraph induced by $V(G)\setminus X$.
Let $G[X]$ denote the subgraph of $G$ induced by $X$. 
We sometimes write $\chi(X)$ for $\chi(G[X])$ and $\omega(X)$ for $\omega(G[X])$.

%Sometimes we write \(X \perp Y\) to mean that \(X\) is anticomplete to \(Y\).
%For an edge \(xy\) with \(x \in X\) and \(y \in Y\), we call \(x\) the \(X\)-\textit{endpoint} and \(y\) the \(Y\)-\textit{endpoint} of this edge.

%A vertex subset \(K \subseteq V(G)\) is a \textit{clique cutset} if \(K\) is a clique and \(G\setminus K\) has more components than \(G\).

We use the following chromatic quasirandomness lemma from~\cite[Lemma~2.6]{Nguyen2025}.

\begin{lemma}[\cite{Nguyen2025}]\label{lem:quasirandom}
	Let $\epsilon\in(0,1)$ and $\delta\in(0,2^{-7}\epsilon^2]$. Every $P_5$-free graph $G$ contains
	either
	\begin{enumerate}[label=\textup{(\roman*)}]
		\item a pure pair $(A,B)$ with $\chi(A),\chi(B)\geq\delta\cdot\chi(G)$; or
		\item an $(\epsilon,\chi)$-dense induced subgraph $F$ with
		$\chi(F)\geq\delta\cdot\chi(G)$.
	\end{enumerate}
\end{lemma}

The next elementary observation is the induced-path chase that underlies the cutset arguments; see also \cite[Lemma~2.7]{Nguyen2025}.

\begin{lemma}[\cite{Nguyen2025}]\label{lem:pure-to-one}
	Let $G$ be $P_5$-free, and let $A,B\subseteq V(G)$ be nonempty and anticomplete such that
	$G[A]$ and $G[B]$ are connected. Every vertex in $V(G)\setminus(A\cup B)$ is pure to at least
	one of $A$ and $B$.
\end{lemma}
%\begin{proof}
%	Suppose that a vertex $v\notin A\cup B$ is mixed on both sets. Since $G[A]$ is connected,
%	there is an edge $aa'$ of $G[A]$ such that $v$ is adjacent to $a$ and nonadjacent to $a'$.
%	Similarly, $G[B]$ has an edge $bb'$ such that $v$ is adjacent to $b$ and nonadjacent to $b'$.
%	Then $a'-a-v-b-b'$ is an induced $P_5$, a contradiction.
%\end{proof}

For a vertex $v\in V(G)$, let $N_G(v)$ denote the set of neighbors of $v$, and let $N_G[v]:=N_G(v)\cup\{v\}$.
When the graph is clear from the context, we omit the subscript.
For real number $\varepsilon>0$, a graph $G$ is \textit{$(\varepsilon,\chi)$-dense} if \(\chi(G\setminus N_G[v])<\varepsilon\cdot\chi(G)\) for every \(v\in V(G)\).
We also use the following simple consequence of chromatic density \cite[Lemma~5.7]{Nguyen2025}.

\begin{lemma}[\cite{Nguyen2025}]\label{lem:dense-clique}
	If $G$ is $(\epsilon,\chi)$-dense, then \(\omega(G)\geq \min\{\epsilon^{-1},\chi(G)\}\).
\end{lemma}

We now record some numerical constants that will be used throughout the paper.
Let
\begin{equation}\label{eq:constants}
	c:=\frac{2}{101},\quad
	\alpha:=\frac{141}{500},\quad
	\beta:=\frac{13}{25},\quad
	\gamma:=\frac{17}{10},\quad\lambda:=\frac{1-c^2}{128},
	\quad\text{and}\quad
	\mu :=(1-2c^4)\lambda.
\end{equation}
%and define
%\begin{equation}\label{eq:lambda-mu}
%	\lambda:=\frac{1-c^2}{128},
%	\quad
%	\mu :=(1-2c^4)\lambda.
%\end{equation}

%\begin{lemma}\label{lem:constants}
%	With $c,\alpha,\beta,\gamma,\lambda,\mu$ as defined in \eqref{eq:constants}, the following inequalities hold:	
%\end{lemma}
%\begin{proof}
	None of the following comparisons relies on rounded decimal values.
	First,
	\[
	\alpha^2-4c
	=\frac{7\,981}{25\,250\,000}>0,
	\]
	which implies
	\begin{equation}\label{ineq:const1}
		\alpha c^{-1/2}\geq 3.
	\end{equation}
	Moreover, $\sqrt2<283/200$ and $\sqrt c<141/1000$. %indeed, $283^2>2\cdot200^2$ and $141^2\cdot101>2\cdot1000^2$. 
	Hence
	\[
	\frac{1}{2\sqrt2}-\frac{\sqrt c}{2}-\alpha
	>\frac{100}{283}-\frac{141}{2000}-\frac{141}{500}
	=\frac{97}{113\,200}>0,
	\]
	which implies
	\begin{equation}\label{ineq:const2}
		\frac{1}{2\sqrt2}-\frac{\sqrt c}{2}>\alpha.
	\end{equation}
	We also have
	\begin{equation}\label{ineq:const3}
		\beta-c>\frac12.
	\end{equation}
	and
	\begin{equation}\label{ineq:const4}
		\sqrt2+\frac{\sqrt c}{\beta}<\gamma.
	\end{equation}
	Furthermore,
	$1-\gamma\alpha-\beta=3/5000>0$, which implies
	\begin{equation}\label{ineq:const5}
		1-\gamma\alpha>\beta.
	\end{equation}
	Since $c<1/32<(13/25)^5$, by taking fifth roots, we have
	\begin{equation}\label{ineq:const6}
		c^{1/5}<\beta.
	\end{equation}
	Observe that $27\alpha^2>2$ and $\beta>1/2$; therefore
	\begin{equation}\label{ineq:const7}
		\beta\alpha^{12/5}3^{18/5}\geq 2.
	\end{equation}
	
	We next record a convenient lower bound for $\mu$.
	Since \[
	\mu c^3-2^{-24}
	=\frac{18\,962\,477\,641\,534\,267}
	{18\,348\,994\,055\,844\,422\,658\,031\,616}>0,
	\]
	we have
	\begin{equation}\label{ineq:const8}
		\mu c^3>2^{-24}.
	\end{equation}
	Since $c<1/50$, we have
	\[
	c^2+2c^4<\frac1{2500}+\frac1{3\,125\,000}<\frac1{2000}.
	\]
	Consequently,
	\begin{equation}\label{eq:mu-lower}
		\mu=\frac{(1-c^2)(1-2c^4)}{128}
		>\frac{1999}{256\,000}>\frac1{129}.
	\end{equation}
	Using \eqref{eq:mu-lower} and $c<1/50$, we obtain
	\[
	\mu^{5/12}c^{-1/3}
	>129^{-5/12}50^{1/3}>\frac13>\alpha.
	\]
	For the middle inequality, raise both sides to the twelfth power and note that
	$50^4 3^{12}=1350^4>129^5$. This implies  
	\begin{equation}\label{ineq:const9}
		\mu^{5/12}c^{-1/3}>\alpha.
	\end{equation}
	Also, $\alpha>7/25$ and
	\[
	\left(\frac{63}{25}\right)^6
	>\left(\frac52\right)^6>129,
	\]
	so \eqref{eq:mu-lower} gives
	\begin{equation}\label{ineq:const10}
		\mu\alpha^6>3^{-12}.
	\end{equation}
	In addition, \eqref{eq:mu-lower} and $c<1/50$ imply
	\begin{equation}\label{ineq:const11}
		\mu>1/129>50^{-18}>c^{18}.
	\end{equation}
	Finally, since 
	\[
	\lambda=\frac{1-c^2}{128}=\frac{1-\frac{4}{10201}}{128}<\frac{1}{128},
	\]
	we have
	\begin{equation}\label{ineq:lambda1}
		\lambda(1+c)\leq 2,\quad 2\lambda c\leq 2,\quad \text{and}\quad 2c<\frac14.
	\end{equation}
%	\begin{equation}\label{ineq:lambda2}
%		2\lambda c\leq 2,
%	\end{equation}
%	and
%	\begin{equation}\label{ineq:c-bound}
%		2c<\frac14.
%	\end{equation}
%\end{proof}

\begin{remark}\label{rmk:numeric}
	For reference, the numerical values closest to equality in~\eqref{ineq:const1},\eqref{ineq:const2},\eqref{ineq:const8},\eqref{ineq:const9}, and \eqref{ineq:const10} are
	\begin{align*}
		\alpha c^{-1/2}&=2.0039865269\ldots,\\
		\frac{1}{2\sqrt2}-\frac{\sqrt c}{2}&=0.2831936361\ldots>0.282=\alpha,\\
		\mu c^3&=6.0638078956\cdot10^{-8}>5.9604644775\cdot10^{-8}=2^{-24},\\
		\mu^{5/12}c^{-1/3}&=0.4894277762\ldots>0.282=\alpha,\\
		\mu\alpha^6&=3.9274821370\cdot10^{-6}>1.8816764232\cdot10^{-6}=3^{-12}.
	\end{align*}
	These values are included only for intuition; they play no role in the proof.
\end{remark}

\section{A refined decomposition along anticomplete pairs}\label{sec:decomp}

The first improvement comes from revisiting the central-cutset decomposition in
\cite[Section~4]{Nguyen2025}.
The extra $C_5$-free assumption reduces a two-set cover to a one-set cover.
For real numbers $0\leq p\leq q$, a graph $G$ is \textit{$(p,q)$-sparse} if every induced
subgraph $F$ of $G$ with $\chi(F)\geq q$ contains an anticomplete pair $(X,Y)$ satisfying
$\chi(X),\chi(Y)\geq p$.

%\subsection{The central-cutset lemma}

\begin{lemma}\label{lem:cutset}
	Let $G$ be a $\{P_5,C_5\}$-free graph, let $\epsilon\in(0,1)$ and $\theta>0$, and let
	$0<p\leq q\leq \chi(G)-\theta-\theta/\epsilon$.
	If $G$ is $(p,q)$-sparse, then $G$ contains at least one of the following:
	\begin{enumerate}[label=\textup{(\roman*)}]
		\item an anticomplete pair $(A,B)$ such that \(\chi(A)\geq q-2\theta\) and \(\chi(B)>\chi(G)-\theta-\theta/\epsilon\);
		\item a complete pair $(X,Y)$ such that $\chi(X)\geq\theta$ and $\chi(Y)\geq p$;
		\item an $(\epsilon,\chi)$-dense induced subgraph $F$ such that
		$\chi(F)\geq\theta/\epsilon$.
	\end{enumerate}
\end{lemma}
\begin{proof}
	A component of $G$ with chromatic number $\chi(G)$ inherits all hypotheses, and any outcome inside that component is also an outcome in $G$.
We may therefore assume that $G$ is connected.
	Suppose that outcomes~\textup{(ii)} and~\textup{(iii)} do not occur.
%	We first record the cutset supplied by sparsity.
	
	\begin{claim}\label{clm:initial-cutset}
		Every connected induced subgraph $F$ of $G$ with $\chi(F)\geq q$ has a minimal nonempty
		cutset $Z$ separating two sets $A,B$ such that $G[A]$ and $G[B]$ are components of
		$F\setminus Z$, \(\chi(A)\geq \max\{p,q-2\theta\}\), and \(\chi(B)\geq p\).
	\end{claim}
	\begin{proof}
		Since $G$ is $(p,q)$-sparse, $F$ contains an anticomplete pair whose two sides have chromatic
		number at least $p$. Replacing each side by a component of the same chromatic number, choose
		such a pair $(A,B)$ with $\chi(A)\geq\chi(B)$ so that $\chi(A)+\chi(B)$ is maximal and,
		subject to this, $|A|+|B|$ is maximal. Because $F$ is connected, there is a minimal nonempty
		cutset $Z$ separating $A$ and $B$. If $A$ were properly contained in its component of
		$F\setminus Z$, then replacing $A$ by that component would preserve all preceding conditions
		and increase $|A|+|B|$; the same argument applies to $B$. Hence $G[A]$ and $G[B]$ are
		components of $F\setminus Z$.
		If some $z\in Z$ had no neighbor in one of these two components, then $Z\setminus\{z\}$ would still separate $A$ and $B$, contrary to the minimality of $Z$.
		It follows that every vertex of $Z$ has a neighbor in each of $A$ and $B$.
		By Lemma~\ref{lem:pure-to-one}, each vertex of
		$Z$ is complete to $A$ or complete to $B$. Let
		$Z_A$ be the set of vertices of $Z$ complete to $A$, and put $Z_B:=Z\setminus Z_A$; then
		$Z_B$ is complete to $B$. Outcome~\textup{(ii)} does not occur, so
		$\chi(Z_A),\chi(Z_B)<\theta$, and hence $\chi(Z)<2\theta$.
		No component $C$ of $F\setminus Z$ can satisfy $\chi(C)>\chi(A)$, because then $(C,B)$
		would be an anticomplete pair with larger value of $\chi(A)+\chi(B)$. Thus $A$ has maximum
		chromatic number among the components of $F\setminus Z$, so $\chi(A)=\chi(F\setminus Z)$,
		and \(\chi(A)\geq\chi(F)-\chi(Z)>\chi(F)-2\theta\geq q-2\theta\).
		Together with $\chi(A),\chi(B)\geq p$, this proves Claim~\ref{clm:initial-cutset}.
	\end{proof}
	
	Apply Claim~\ref{clm:initial-cutset} to $F=G$, and let $D$ be the resulting minimal
	cutset. Retain as $A$ the component supplied by Claim~\ref{clm:initial-cutset}. List the components of
	$G\setminus D$ other than $A$ whose chromatic number is at least $p$ as
	$B_1,\ldots,B_k$, and let $E$ be the union of all remaining components. The component
	labelled $B$ in Claim~\ref{clm:initial-cutset} appears among the $B_i$, so $k\geq1$. We have
	obtained a partition \((A,D,B_1,\ldots,B_k,E)\) of $V(G)$, where $E$ is allowed to be empty, with the following properties:
	\begin{itemize}
		\item $D$ is nonempty, and the sets $A,B_1,\ldots,B_k,E$ are pairwise anticomplete;
		\item $G[A]$ and $G[B_i]$ are connected for every $i\in[k]$;
		\item every vertex of $D$ has a neighbor in $B_1\cup\cdots\cup B_k$;
		\item $\chi(A)\geq\max\{p,q-2\theta\}$, $\chi(E)<p$, and
		$\chi(B_i)\geq p$ for every $i\in[k]$.
	\end{itemize}
	Among all partitions with these properties, choose one with $k$ maximum.
	
	\begin{claim}\label{clm:A-below-q}
		$\chi(A)<q$.
	\end{claim}
	\begin{proof}
		Suppose that $\chi(A)\geq q$. Apply Claim~\ref{clm:initial-cutset} to $G[A]$. It gives a
		minimal nonempty cutset $D'$ in $G[A]$ and two components $A',B'$ of $G[A]\setminus D'$
		with \(\chi(A')\geq\max\{p,q-2\theta\}\), and \(\chi(B')\geq p\).
		List all components of $G[A]\setminus(A'\cup D')$ with chromatic number at least $p$ as
		$B_{k+1},\ldots,B_{k+r}$; here $r\geq1$ because $B'$ is one of them. Let $E'$ be the union of
		the remaining components, so $\chi(E')<p$. The sets $E$ and $E'$ are anticomplete, and hence
		$\chi(E\cup E')<p$. Every vertex of $D$ still has a neighbor in one of the original blocks,
		while every vertex of $D'$ has a neighbor in the component $B'$ and therefore in one of the
		new blocks. Consequently, replacing $A$ by $A'$, replacing $D$ by $D\cup D'$, replacing $E$
		by $E\cup E'$, and adjoining $B_{k+1},\ldots,B_{k+r}$ produces a partition of the same type
		with more than $k$ blocks. This contradicts the maximality of $k$, proving Claim~\ref{clm:A-below-q}.
	\end{proof}
	
	Let $S$ be the set of vertices of $D$ that are mixed on $A$. Since $A$ is anticomplete to each
	$B_i$, Lemma~\ref{lem:pure-to-one} implies that every vertex of $S$ is pure to every $B_i$.
	The following two claims control $G[S]$.
	
	\begin{claim}\label{clm:crossing}
		Let $u,v\in S$ and $i,j\in[k]$. If $v$ is complete to $B_i$ and anticomplete to $B_j$, while
		$u$ is complete to $B_j$ and anticomplete to $B_i$, then $u$ and $v$ are adjacent.
	\end{claim}
	\begin{proof}
		Suppose for a contradiction that $u$ and $v$ are nonadjacent, and choose $z_i\in B_i$ and
		$z_j\in B_j$. Since $G[A]$ is connected and both $u$ and $v$ have neighbors in $A$, a
		shortest $u$--$v$ path in $G[A\cup\{u,v\}]$ has the form \(P=v-p_1-\cdots-p_\ell-u\) with $\ell\geq1$ and $p_h\in A$ for every $h$; this path is induced.
		Now $z_i$ is adjacent to $v$
		and nonadjacent to $u$, while $z_j$ is adjacent to $u$ and nonadjacent to $v$.
		Since both
		$B_i$ and $B_j$ are anticomplete to $A$, and $B_i$ is anticomplete to $B_j$, the path
		\(z_i-v-p_1-\cdots-p_\ell-u-z_j\) is induced.
		This path has at least five vertices, so five consecutive vertices induce a $P_5$, a contradiction.
		This proves Claim~\ref{clm:crossing}.
	\end{proof}
	
	\begin{claim}\label{clm:one-cover}
		For every $v\in S$, \(\chi(S\setminus N[v])<\theta\).
	\end{claim}
	\begin{proof}
		Let $I$ be the set of indices $i\in[k]$ for which $v$ has a neighbor in $B_i$.
		Since every vertex of $D$ has a neighbor in $B_1\cup\cdots\cup B_k$, the set $I$ is
		nonempty, and by Lemma~\ref{lem:pure-to-one}, $v$ is complete to every $B_i$ with $i\in I$. For $i\in I$, let $S_i$ be the set of vertices in $S\setminus N[v]$ that have a neighbor in $B_i$.
		By Lemma~\ref{lem:pure-to-one}, each $S_i$ is complete to $B_i$, and therefore $\chi(S_i)<\theta$, since outcome~\textup{(ii)} does not occur.
		
		We show that the sets $S_i$ cover $S\setminus N[v]$. Let
		$u\in S\setminus N[v]$.
		If $u$ had no neighbor in any $B_i$ with $i\in I$, then Lemma~\ref{lem:pure-to-one} and the defining
		property of $D$ would give an index $j\notin I$ such that $u$ is complete to $B_j$. Choosing any
		$i\in I$, Claim~\ref{clm:crossing} would force $u$ and $v$ to be adjacent, a contradiction.
		Thus \(S\setminus N[v]\subseteq\bigcup_{i\in I}S_i\).
		Choose $J\subseteq I$ minimal with this covering property. We claim that $|J|\leq1$. Suppose
		that distinct $i,j\in J$ exist.
		By minimality, choose
		\(u_i\in S_i\setminus S_j\) and \(u_j\in S_j\setminus S_i\).
		Then by Lemma~\ref{lem:pure-to-one}, $u_i$ is complete to $B_i$ and anticomplete to $B_j$, while $u_j$ is complete to $B_j$ and anticomplete to $B_i$. Claim~\ref{clm:crossing} gives $u_i u_j\in E(G)$. For arbitrary
		$z_i\in B_i$ and $z_j\in B_j$, the five vertices \(z_i,v,z_j,u_j,u_i\) in this cyclic order induce a $C_5$, a contradiction.
		%: the required cycle edges follow from the definitions, and all five possible chords are absent because $B_i$ and $B_j$ lie in different components of $G\setminus D$, both $u_i,u_j$ lie in $S\setminus N[v]$, and each private witness is anticomplete to the other block. This contradicts the $C_5$-free hypothesis. 
		Hence $|J|\leq1$, and Claim~\ref{clm:one-cover} follows from $\chi(S_i)<\theta$.
	\end{proof}
	
	If $\chi(S)\geq\theta/\epsilon$, then Claim~\ref{clm:one-cover} gives, for every $v\in S$, \(\chi(S\setminus N[v])<\theta\leq\epsilon\cdot\chi(S)\).
	Thus $G[S]$ is $(\epsilon,\chi)$-dense and has chromatic number at least $\theta/\epsilon$, contrary
	to the failure of outcome~\textup{(iii)}. Therefore
	\begin{equation}\label{eq:S-small}
		\chi(S)<\theta/\epsilon.
	\end{equation}
	
	Let $R$ be the set of vertices in $D$ complete to $A$. Since $\chi(A)\geq p$ and
	outcome~\textup{(ii)} fails,
	\begin{equation}\label{eq:R-small}
		\chi(R)<\theta.
	\end{equation}
	By Lemma~\ref{lem:pure-to-one}, every vertex in $D\setminus(R\cup S)$ is anticomplete to $A$. Also, $A$ is anticomplete to
	$B_1\cup\cdots\cup B_k\cup E$. Set
	\[
	T:=(D\setminus(R\cup S))\cup B_1\cup\cdots\cup B_k\cup E.
	\]
	By \eqref{eq:S-small} and \eqref{eq:R-small},
	\[
	\chi(A\cup T)
	\geq \chi(G)-\chi(R\cup S)
	>\chi(G)-\theta-\theta/\epsilon
	\geq q.
	\]
	Claim~\ref{clm:A-below-q} gives $\chi(A)<q$. Since $A$ and $T$ are anticomplete,
	$\chi(A\cup T)=\max\{\chi(A),\chi(T)\}$, so \(\chi(T)>\chi(G)-\theta-\theta/\epsilon\).
	The pair $(A,T)$ satisfies outcome~\textup{(i)} and completes the proof of Lemma~\ref{lem:cutset}.
\end{proof}

%\subsection{The improved one-step decomposition}

\begin{lemma}\label{lem:base-decomp}
	For every $\epsilon\in(0,c]$, every $\{P_5,C_5\}$-free graph $G$ with $\chi(G)>0$ contains at
	least one of the following:
	\begin{enumerate}[label=\textup{(\roman*)}]
		\item an anticomplete pair $(A,B)$ with \(\chi(A),\chi(B)\geq(1-2\epsilon^2)\chi(G)\);
		\item a complete pair $(X,Y)$ with \(\chi(X),\chi(Y)\geq\lambda\epsilon^3\chi(G)\);
		\item an $(\epsilon,\chi)$-dense induced subgraph $F$ with \(\chi(F)\geq\lambda\epsilon^2\chi(G)\).
	\end{enumerate}
\end{lemma}
\begin{proof}
	Set \(\theta:=\lambda\epsilon^3\chi(G)\), \(p:=\theta\), and \(q:=(1-\epsilon^2)\chi(G)\).
	The inequality in \eqref{ineq:lambda1} implies \(\theta+\theta/\epsilon=\lambda\epsilon^2(1+\epsilon)\chi(G)\leq\epsilon^2\chi(G)\), so $q\leq\chi(G)-\theta-\theta/\epsilon$.
	Suppose that outcomes~\textup{(ii)} and~\textup{(iii)} do not occur. We claim that $G$ is
	$(p,q)$-sparse. Let $H$ be an induced subgraph with $\chi(H)\geq q$. Apply
	Lemma~\ref{lem:quasirandom} to $H$ with parameter $\epsilon$ and
	$\delta=2^{-7}\epsilon^2$. Since $\epsilon\leq c$, we have
	\[
	\delta\chi(H)
	\geq \frac{\epsilon^2}{128}(1-\epsilon^2)\chi(G)
	\geq \frac{\epsilon^2}{128}(1-c^2)\chi(G)
	=\lambda\epsilon^2\chi(G).
	\]
	If the dense outcome of Lemma~\ref{lem:quasirandom} occurs, then outcome~\textup{(iii)} above
	holds. If its pure pair is complete, then outcome~\textup{(ii)} holds, because
	$\lambda\epsilon^2\chi(G)\geq p$. Therefore the pure pair must be anticomplete, and its two sides have chromatic number at least $p$. This proves that $G$ is $(p,q)$-sparse.
	Apply Lemma~\ref{lem:cutset}. Its second and third outcomes are exactly
	outcomes~\textup{(ii)} and~\textup{(iii)} here. In its first outcome, \(\chi(A)\geq(1-\epsilon^2-2\lambda\epsilon^3)\chi(G)
	\geq(1-2\epsilon^2)\chi(G)\), where we used $2\lambda\epsilon<1$ in \eqref{ineq:lambda1}.
	Also, \(\chi(B)>\bigl(1-\lambda\epsilon^2(1+\epsilon)\bigr)\chi(G)
	\geq(1-2\epsilon^2)\chi(G)\) by \eqref{ineq:lambda1}.
	Thus outcome~\textup{(i)} holds.
	This proves Lemma~\ref{lem:base-decomp}.
\end{proof}

%\subsection{Iteration}

For $s\geq0$, we define
\[
\phi_0(c):=1,\quad
\phi_s(c):=\prod_{i=1}^s\bigl(1-c^{2^{i+1}}\bigr),
\]
and, for $0\leq r\leq s$, define
\[
\phi_{r,s}(c):=\frac{\phi_s(c)}{\phi_r(c)}
=\prod_{r<i\leq s}\bigl(1-c^{2^{i+1}}\bigr).
\]
The empty product equals $1$.

\begin{lemma}\label{lem:phi-bound}
	For all integers $0\leq r\leq s$, \(\phi_{r,s}(c)\geq1-2c^{2^{r+2}}\).
	In particular, $\phi_s(c)\geq1-2c^4$.
\end{lemma}
\begin{proof}
	The statement is immediate when $r=s$. Otherwise, using
	$\prod_j(1-x_j)\geq1-\sum_j x_j$ for $x_j\in[0,1]$, and setting
	$z=c^{2^{r+2}}$, we obtain
	\[
	\phi_{r,s}(c)
	\geq1-\sum_{r<i\leq s}c^{2^{i+1}}
	\geq1-(z+z^2+z^4+\cdots).
	\]
	Since $z\leq c^4<1/2$ by \eqref{ineq:lambda1}, the final sum is at most $z/(1-z)\leq2z$.
	This proves Lemma~\ref{lem:phi-bound}.
\end{proof}

\begin{lemma}\label{lem:iterated-decomp}
	For every integer $s\geq0$, every $\{P_5,C_5\}$-free graph $G$ contains at least one of the
	following:
	\begin{enumerate}[label=\textup{(\roman*)}]
		\item an anticomplete pair $(A,B)$ with \(\chi(A),\chi(B)\geq\bigl(1-2c^{2^{s+1}}\bigr)\chi(G)\);
		\item a complete pair $(P,Q)$ with \(\chi(P),\chi(Q)\geq\phi_s(c)\lambda c^3\chi(G)\);
		\item for some integer $r$ with $1\leq r\leq s$, a complete pair $(X,Y)$ with \(\chi(X)\geq\phi_{r,s}(c)\lambda c^{3\cdot2^r}\chi(G)\) and \(\chi(Y)\geq\phi_{r,s}(c)\bigl(1-3c^{2^r}\bigr)\chi(G)\);
		\item for some integer $r$ with $0\leq r\leq s$, a
		$(c^{2^r},\chi)$-dense induced subgraph $F$ with \(\chi(F)\geq\phi_{r,s}(c)\lambda c^{2^{r+1}}\chi(G)\).
	\end{enumerate}
\end{lemma}
\begin{proof}
	We use induction on $s$. The case $s=0$ is Lemma~\ref{lem:base-decomp} with $\epsilon=c$.
	
	Assume the statement for $s$, and let \(e:=c^{2^{s+1}}\), \(p:=(1-3e)\chi(G)\), and \(q:=(1-e^2)\chi(G)\).
	Suppose that outcomes~\textup{(ii)--(iv)} for $s+1$ do not occur. We show that $G$ is
	$(p,q)$-sparse. Let $H$ be an induced subgraph with $\chi(H)\geq q$. Apply the induction
	hypothesis to $H$. If its second outcome occurs, then \(\phi_s(c)\lambda c^3\chi(H)
	\geq \phi_s(c)(1-e^2)\lambda c^3\chi(G)
	=\phi_{s+1}(c)\lambda c^3\chi(G)\), which is the second outcome for $G$ at level $s+1$.
	For every $r\leq s$, the identity \(\phi_{r,s}(c)(1-e^2)=\phi_{r,s+1}(c)\) transfers the third and fourth outcomes in exactly the same way. Each possibility contradicts
	our assumption that outcomes~\textup{(ii)--(iv)} fail for $G$.
	Hence the first outcome occurs
	in $H$, and its two sides
	have chromatic number at least \((1-2e)\chi(H)
	\geq(1-2e)(1-e^2)\chi(G)
	\geq(1-3e)\chi(G)=p\), because $(1-2e)(1-e^2)-(1-3e)=e-e^2+2e^3>0$.
	Thus $G$ is $(p,q)$-sparse.
	
	Apply Lemma~\ref{lem:cutset} with density parameter $e$ and
	$\theta=\lambda e^3\chi(G)$. The condition
	$q\leq\chi(G)-\theta-\theta/e$ follows from $\lambda(1+c)<1$ in \eqref{ineq:lambda1}. The second outcome of
	Lemma~\ref{lem:cutset} gives outcome~\textup{(iii)} here with $r=s+1$, and its third outcome
	gives outcome~\textup{(iv)} with $r=s+1$. In the remaining case, the two sides of the
	anticomplete pair have chromatic number at least \((1-e^2-2\lambda e^3)\chi(G)\geq(1-2e^2)\chi(G)\) and \(\bigl(1-\lambda e^2(1+e)\bigr)\chi(G)
	\geq(1-2e^2)\chi(G)\), respectively. This is outcome~\textup{(i)} for $s+1$.
	This proves Lemma~\ref{lem:iterated-decomp}.
\end{proof}

%The iteration yields the outer decomposition needed later.

\begin{lemma}\label{lem:outer-decomp}
	Let $G$ be a $\{P_5,C_5\}$-free graph with $\chi(G)\geq c^{-4}$. Then $G$ contains at least
	one of the following:
	\begin{enumerate}[label=\textup{(\roman*)}]
		\item a complete pair $(P,Q)$ with \(\chi(P),\chi(Q)\geq2^{-24}\chi(G)\);
		\item a complete pair $(X,Y)$ and a number $y\in(0,1/4)$ such that \(\chi(X)\geq y^{24}\chi(G)\) and \(\chi(Y)\geq(1-y)\chi(G)\);
		\item an $\epsilon\in[\chi(G)^{-1/4},c]$ and an $(\epsilon,\chi)$-dense induced subgraph $F$ with \(\chi(F)\geq\mu\epsilon^2\chi(G)\).
	\end{enumerate}
\end{lemma}
\begin{proof}
	Without loss of generality, we may assume that $G$ is connected.
	Put $w:=\chi(G)$. Let $s\geq0$ be maximal such that \(c^{-2^s}\leq w^{1/4}\).
	Such an $s$ exists because $w\geq c^{-4}$. Apply Lemma~\ref{lem:iterated-decomp}.
	If its second outcome occurs, Lemma~\ref{lem:phi-bound} and \eqref{ineq:const8} give \(\chi(P),\chi(Q)
	\geq(1-2c^4)\lambda c^3w
	=\mu c^3w
	>2^{-24}w\), which is outcome~\textup{(i)}.
	Suppose that its third outcome occurs for an integer $r\in[1,s]$, and set
	$y:=c^{2^{r-1}}$.
	Since $c<1/4$ by \eqref{ineq:lambda1}, we have $0<y\leq c<1/4$ and $c^{2^r}=y^2$.
	By Lemma~\ref{lem:phi-bound}, \(\phi_{r,s}(c)\lambda\geq\mu\).
	Therefore \(\chi(X)\geq\mu y^6w\geq y^{24}w\), where the last inequality follows from $y^{18}\leq c^{18}<\mu$ by \eqref{ineq:const11}. Also, \(\chi(Y) \geq(1-2y^8)(1-3y^2)w \geq(1-y)w\); in fact, \((1-2y^8)(1-3y^2)-(1-y) =y\bigl(1-3y-2y^7+6y^9\bigr)>0\),
	because $0<y<1/4$ gives
	$1-3y-2y^7+6y^9>1-3/4-2/4^7>0$.
	Thus outcome~\textup{(ii)} holds.
	Suppose that the fourth outcome of Lemma~\ref{lem:iterated-decomp} occurs for an integer
	$r\in[0,s]$. Put $\epsilon:=c^{2^r}$. Then $\epsilon\leq c$, while \(\epsilon\geq c^{2^s}\geq w^{-1/4}\) by the definition of $s$. Moreover, Lemma~\ref{lem:phi-bound} gives \(\chi(F)\geq\mu\epsilon^2w\).
	This is outcome~\textup{(iii)}.
	
	It remains to consider the first outcome of Lemma~\ref{lem:iterated-decomp}. The maximality of
	$s$ gives $c^{-2^{s+1}}>w^{1/4}$. Since $w^{1/8}\geq c^{-1/2}>2$, we have \(c^{-2^{s+1}}>2w^{1/8}\) and hence \(2c^{2^{s+1}}<w^{-1/8}\).
	Thus $G$ contains an anticomplete pair $(A,B)$ with \(\chi(A),\chi(B)>(1-w^{-1/8})w\).
	Among all such pairs, choose one for which $G[A]$ and $G[B]$ are connected and
	$|A|+|B|$ is maximal. Let $S$ be a minimal nonempty cutset separating $A$ and $B$.
	The maximality of $|A|+|B|$ ensures that $G[A]$ and $G[B]$ are components of
	$G\setminus S$. Every vertex $v\in S$ has a neighbor in each of $A$ and $B$, so
	Lemma~\ref{lem:pure-to-one} implies that $v$ is complete to $A$ or to $B$.
	Consequently, \(\chi(N(v))>(1-w^{-1/8})w\).
	Set $X:=\{v\}$, $Y:=N(v)$, and $y:=w^{-1/8}$. Since $w\geq c^{-4}$, we have
	$y\leq c^{1/2}<1/4$.
	Furthermore, \(\chi(X)=1\geq w^{-2}=y^{24}w\) and \(\chi(Y)>(1-y)w\).
	This is outcome~\textup{(ii)}, proving Lemma~\ref{lem:outer-decomp}.
\end{proof}

\section{A sharper density increment}\label{sec:density}

We now improve the dense part of the argument.
We begin with:
%The first lemma retains the full counting strength of the partition step instead of weakening it to a square-root estimate.

%\subsection{A stronger partition lemma}

\begin{lemma}\label{lem:partition}
	Let $z\in(0,1/2)$, let $G$ be a graph, and let $(A_1,\ldots,A_m)$ be a partition of $V(G)$ into
	nonempty sets satisfying $\chi(A_i)\leq z\cdot\chi(G)$ for every $i\in[m]$. Then there exist pairwise
	disjoint nonempty index sets $I_1,\ldots,I_k\subseteq[m]$ such that \(k\geq\frac{1-z}{2z}\) and \(\chi\left(\bigcup_{i\in I_j}A_i\right)\geq z\cdot\chi(G)\) for every \(j\in[k]\).
\end{lemma}
\begin{proof}
	Choose a partition $J_1\cup\cdots\cup J_s=[m]$ with $s$ minimum subject to \(\chi\left(\bigcup_{i\in J_j}A_i\right)\leq2z\cdot\chi(G)\) for every \(j\in[s]\).
	Such a partition exists by taking all $J_j$ to be singletons. By the minimality of $s$, at most
	one group has chromatic number at most $z\cdot\chi(G)$, because two such groups could be merged.
	Let $k$ be the number of groups with chromatic number greater than $z\cdot\chi(G)$. If all groups
	are of this type, then \(\chi(G)\leq2zk\cdot\chi(G)\), so $k\geq(2z)^{-1}>(1-z)/(2z)$. Otherwise there is one remaining group of chromatic number at
	most $z\cdot\chi(G)$, and \(\chi(G)\leq z\cdot\chi(G)+2zk\cdot\chi(G)\), which gives $k\geq(1-z)/(2z)$.
	Taking the $k$ high-chromatic groups as	$I_1,\ldots,I_k$ proves Lemma~\ref{lem:partition}.
\end{proof}

A \textit{blockade} in $G$ is a sequence $(B_1,\ldots,B_k)$ of pairwise disjoint nonempty
subsets of $V(G)$. Its \textit{length} is $k$, and its \textit{mass} is $\min\{\chi(B_i):~i\in[k]\}$. The blockade is \textit{complete} if every two distinct blocks are
complete to one another. It is \textit{restrained} if \(\sum_{i=1}^k\omega(B_i)\leq \omega(G)\).
Every complete blockade is restrained.
The next result strengthens \cite[Lemma~5.5]{Nguyen2025}.

\begin{lemma}\label{lem:local-blockade}
	Let $t\in(0,1/4)$, let $G$ be $\{P_5,C_5\}$-free, let $v\in V(G)$, and let
	$A\subseteq N(v)$ and $B\subseteq V(G)\setminus N[v]$ be nonempty. Suppose that \(\chi(A\setminus N(u))\leq t\cdot\chi(A)\) for every \(u\in B\).
	Then at least one of the following holds:
	\begin{enumerate}[label=\textup{(\roman*)}]
		\item there is a set $D\subseteq A$ such that \(\chi(D)\geq(1-\sqrt t)\chi(A)\) and $(B,D)$ is restrained;
		\item $G[A]$ contains a complete blockade of length $k$ and mass at least $t\cdot\chi(A)$, where \(k\geq\frac{1-\sqrt t}{2\sqrt t}\).
	\end{enumerate}
\end{lemma}
\begin{proof}
	Let $K$ be a maximum clique in $G[B]$, and let $S$ be the set of vertices of $A$ having a
	nonneighbor in $K$. If $\chi(S)<\sqrt t\,\chi(A)$, then $D:=A\setminus S$ satisfies the
	chromatic lower bound in outcome~\textup{(i)}. Since $K$ is complete to $D$, \(\omega(B)+\omega(D)=|K|+\omega(D)\leq\omega(G)\), so $(B,D)$ is restrained.
	We may therefore assume that $\chi(S)\geq\sqrt t\,\chi(A)$. Choose
	$v_1,\ldots,v_\ell\in K$ with $\ell$ minimum such that every vertex of $S$ is nonadjacent to at
	least one $v_i$. For each $i\in[\ell]$, let $A_i$ consist of the vertices of $S$ that are
	nonadjacent to $v_i$ and adjacent to $v_1,\ldots,v_{i-1}$. Then
	$(A_1,\ldots,A_\ell)$ is a partition of $S$ into nonempty sets, and \(\chi(A_i)\leq t\cdot\chi(A)\leq\sqrt t\,\chi(S)\) for every \(i\in[\ell]\).
	Apply Lemma~\ref{lem:partition} to $G[S]$ with $z=\sqrt t$. We obtain disjoint index sets
	$I_1,\ldots,I_k$ with \(k\geq\frac{1-\sqrt t}{2\sqrt t}\) such that, for $C_j:=\bigcup_{i\in I_j}A_i$, \(\chi(C_j)\geq\sqrt t\,\chi(S)\geq t\cdot\chi(A)\) for every \(j\in[k]\).
	
	It remains to show that the blocks $C_1,\ldots,C_k$ are pairwise complete. Suppose that
	nonadjacent vertices $u_1\in C_{j_1}$ and $u_2\in C_{j_2}$ exist for distinct $j_1,j_2$. Choose
	$i_1\in I_{j_1}$ and $i_2\in I_{j_2}$ with $u_1\in A_{i_1}$ and $u_2\in A_{i_2}$, and assume
	$i_1<i_2$. Then $u_2$ is adjacent to $v_{i_1}$.
	If $u_1$ is nonadjacent to $v_{i_2}$, then \(u_1-v-u_2-v_{i_1}-v_{i_2}\) gives an induced $P_5$.
	If $u_1$ is adjacent to $v_{i_2}$, the same five vertices in cyclic order
	give an induced $C_5$.
	Both cases lead to a contradiction.
	Hence the blockade is complete, and outcome~\textup{(ii)} holds.
	This proves Lemma~\ref{lem:local-blockade}.
\end{proof}

%\subsection{The density-increment step}

\begin{lemma}\label{lem:density-step}
	Let $y\in(0,c]$, and let $G$ be a $(y,\chi)$-dense $\{P_5,C_5\}$-free graph. Then at least one
	of the following holds:
	\begin{enumerate}[label=\textup{(\roman*)}]
		\item $G$ contains a $(y^{6/5},\chi)$-dense induced subgraph $H$ with \(\chi(H)\geq\beta\cdot\chi(G)\);
		\item $G$ contains a restrained blockade of length $k$ and mass at least $k^{-6}\chi(G)$,
		where \(k\geq\alpha y^{-1/2}\).
	\end{enumerate}
\end{lemma}
\begin{proof}
	Let $m\geq0$ be maximal such that $G$ contains a restrained blockade
	$(B_0,B_1,\ldots,B_m)$ satisfying
	\begin{align}
		\chi(B_m)&\geq(1-\gamma\sqrt y)^m\chi(G), \label{eq:terminal-block}\\
		\chi(B_i)&\geq\beta y^{6/5}\chi(G)
		\quad\text{for every }0\leq i<m. \label{eq:stored-blocks}
	\end{align}
	The choice $m=0$ and $B_0=V(G)$ shows that such an $m$ exists. Since every block is
	nonempty and the blocks are pairwise disjoint, only finitely many values of $m$ are possible;
	hence a maximum exists.
	Suppose first that $m\geq\alpha y^{-1/2}$. By \eqref{ineq:const1},
	$m\geq\alpha c^{-1/2}>2$, so $m\geq3$. The blockade
	$(B_0,\ldots,B_{m-1})$ has length $m$. Since $m\geq\alpha y^{-1/2}$,
	$y\geq\alpha^2m^{-2}$, and therefore \(\beta y^{6/5}
	\geq\beta\alpha^{12/5}m^{-12/5}
	\geq m^{-6}\).
	For the last inequality, it is enough to verify
	$\beta\alpha^{12/5}m^{18/5}\geq1$; this follows from $m\geq3$ and
	\eqref{ineq:const7}. Hence outcome~\textup{(ii)} holds.
	
	We may now assume that $m<\alpha y^{-1/2}$. Since $y\leq c$ and
	$\gamma\sqrt c<(17/10)(141/1000)=2397/10\,000<1$, Bernoulli's inequality ($(1-x)^m\ge 1-mx$ for $x\in[0,1]$), together with \eqref{ineq:const5} and \eqref{eq:terminal-block}, gives
	\[
	\chi(B_m)
	\geq(1-\gamma\sqrt y)^m\chi(G)
	\geq(1-m\gamma\sqrt y)\chi(G)
	>(1-\gamma\alpha)\chi(G)
	>\beta\chi(G).
	\]
	If $G[B_m]$ is $(y^{6/5},\chi)$-dense, then outcome~\textup{(i)} holds. Otherwise choose
	$v\in B_m$ such that \(\chi(B_m\setminus N[v])\geq y^{6/5}\chi(B_m)\).
	Set \(B:=B_m\setminus N[v]\) and \(A:=B_m\cap N(v)\).
	The set $B$ is nonempty. Since $G$ is $(y,\chi)$-dense, \(\chi(B)<y\cdot\chi(G)<\frac{y}{\beta}\chi(B_m)\).
	As $v$ is anticomplete to $B$, the set $B\cup\{v\}$ has chromatic number $\chi(B)$, and hence
	\begin{equation}\label{eq:A-large}
		\chi(A)\geq\chi(B_m)-\chi(B)
		>\left(1-\frac{y}{\beta}\right)\chi(B_m).
	\end{equation}
	Moreover, for every $u\in B$, \(\chi(A\setminus N(u))<y\cdot\chi(G)
	<\frac{y}{\beta-y}\chi(A)
	<2y\cdot\chi(A)\), where the last inequality follows from \eqref{ineq:const3}. Apply Lemma~\ref{lem:local-blockade} to the induced graph $G[B_m]$, with
	$t=2y$ and with the sets $A,B$ defined above.
	
	Suppose first that it gives a set $D\subseteq A$ for which $(B,D)$ is restrained in $G[B_m]$
	and \(\chi(D)\geq(1-\sqrt{2y})\chi(A)\).
	Using \eqref{eq:A-large},
	\begin{align*}
		\chi(D)
		&>(1-\sqrt{2y})\left(1-\frac{y}{\beta}\right)\chi(B_m)\\
		&\geq\left(1-\sqrt{2y}-\frac{y}{\beta}\right)\chi(B_m)\\
		&\geq(1-\gamma\sqrt y)\chi(B_m),
	\end{align*}
	where the last step uses \eqref{ineq:const4}. Also, \(\chi(B)\geq y^{6/5}\chi(B_m)>\beta y^{6/5}\chi(G)\).
	Because $(B,D)$ is restrained in $G[B_m]$, replacing the terminal block $B_m$ by the ordered
	pair $B,D$ preserves restraint. The resulting blockade \((B_0,\ldots,B_{m-1},B,D)\) satisfies \eqref{eq:terminal-block} and \eqref{eq:stored-blocks} with $m+1$ in place of $m$.
	This contradicts the maximality of $m$.
	
	Therefore Lemma~\ref{lem:local-blockade} gives a complete blockade in $G[A]$ of length $k$
	and mass at least $2y\cdot\chi(A)$, where
	\[
	k\geq\frac{1-\sqrt{2y}}{2\sqrt{2y}}
	=\left(\frac{1}{2\sqrt2}-\frac{\sqrt y}{2}\right)y^{-1/2}
	>\alpha y^{-1/2}
	\]
	by \eqref{ineq:const2}. Its mass satisfies \(2y\chi(A)>2y(\beta-y)\chi(G)>y\chi(G)\) by \eqref{ineq:const3}. The lower bound on $k$ and \eqref{ineq:const1} imply $k\geq3$ and
	$y>\alpha^2k^{-2}$. Since
	\[
	\alpha^2 3^4=\frac{1\,610\,361}{250\,000}>1,
	\]
	we have $y>k^{-6}$. Hence the blockade has mass greater than $k^{-6}\chi(G)$. Since it is complete, it is also
	restrained, and outcome~\textup{(ii)} follows.
	This proves Lemma~\ref{lem:density-step}.
\end{proof}

We now convert Lemma~\ref{lem:density-step} into the form needed in the final combination. We
use a discrete sequence of density scales. This avoids a minor endpoint issue that would arise from
choosing a minimum over a continuous interval, because the definition of $(y,\chi)$-density uses a
strict inequality.

\begin{lemma}\label{lem:dense-blockade}
	Let $\epsilon\in(0,c]$, and let $G$ be an $(\epsilon,\chi)$-dense $\{P_5,C_5\}$-free graph with $\chi(G)>0$.
	Then $G$ contains a restrained blockade of length $k$ satisfying \(k\geq\min\left\{\alpha\epsilon^{-1/2},\chi(G)^{5/12}\right\}\) and mass at least \(\alpha^2k^{-8}\chi(G)\).
\end{lemma}
\begin{proof}
	Write $w:=\chi(G)$. Suppose first that $w\leq\epsilon^{-2}$. Lemma~\ref{lem:dense-clique} gives \(\omega(G)\geq\min\{\epsilon^{-1},w\}\geq w^{1/2}\).
	Let $K$ be a maximum clique and take its vertices as singleton blocks. This complete blockade has
	length $k=\omega(G)\geq w^{1/2}\geq w^{5/12}$ and mass $1$. Moreover, \(\alpha^2k^{-8}w\leq\alpha^2w^{-3}\leq1\), so the required mass bound holds.
	Assume now that $w>\epsilon^{-2}$, so $w^{-1/2}<\epsilon$. Define a decreasing sequence by \(y_0:=\epsilon\) and \(y_{i+1}:=y_i^{6/5}~(i\geq0)\).
	Let $J$ be the largest integer for which $y_J\geq w^{-1/2}$. Among the indices
	$j\in\{0,1,\ldots,J\}$ for which $G$ has a $(y_j,\chi)$-dense induced subgraph $F$ satisfying
	$\chi(F)\geq y_j w$, choose $j$ maximum, and fix such an $F$. The index $j=0$ is admissible
	because $G$ itself is $(\epsilon,\chi)$-dense.
	
	Apply Lemma~\ref{lem:density-step} to $F$ with $y=y_j$. Suppose that its first outcome occurs.
	Then there is a $(y_j^{6/5},\chi)$-dense induced subgraph $H$ with \(\chi(H)\geq\beta\chi(F)\geq\beta y_j w>y_j^{6/5}w\), where the strict inequality uses $y_j^{1/5}\leq c^{1/5}<\beta$ from
	\eqref{ineq:const6}. If $j<J$, this contradicts the maximality of $j$. Hence $j=J$, and by the
	definition of $J$, we have \(y_j^{6/5}<w^{-1/2}\), so \(y_j<w^{-5/12}\).
	Lemma~\ref{lem:dense-clique} applied to $F$ gives
	\[
	\omega(F)\geq\min\{y_j^{-1},\chi(F)\}
	\geq\min\{y_j^{-1},y_j w\}
	>w^{5/12}.
	\]
	Here $y_j^{-1}>w^{5/12}$ follows from $y_j<w^{-5/12}$, while
	$y_j w\geq w^{1/2}>w^{5/12}$ follows from $y_j\geq w^{-1/2}$. Taking a maximum clique of
	$F$ as singleton blocks gives a complete blockade of length $k>w^{5/12}$. This blockade is
	restrained in $G$, because its vertices form a clique in $G$. Its mass is $1$, while \(\alpha^2k^{-8}w<\alpha^2w^{-7/3}<1\).
	Thus the conclusion holds.
	
	It remains to consider the second outcome of Lemma~\ref{lem:density-step}. It gives a blockade
	that is restrained in $F$, and hence also restrained in $G$, of length
	$k\geq\alpha y_j^{-1/2}\geq\alpha\epsilon^{-1/2}$ and mass at least $k^{-6}\chi(F)$. Since
	$\chi(F)\geq y_j w$ and
	$k\geq\alpha y_j^{-1/2}$, we have $y_j\geq\alpha^2k^{-2}$ and therefore \(k^{-6}\chi(F)
	\geq k^{-6}y_j w
	\geq\alpha^2k^{-8}w\).
	This completes the proof of Lemma~\ref{lem:dense-blockade}.
\end{proof}

\section{The structural theorem and the \texorpdfstring{$\chi$}{chi}-bound}\label{sec:main-proof}

We first establish the complete-pair--or--blockade statement that directly supports induction on
clique number.

\begin{theorem}\label{thm:structure24}
	Every $\{P_5,C_5\}$-free graph $G$ with $\chi(G)\geq2$ contains at least one of the following:
	\begin{enumerate}[label=\textup{(\roman*)}]
		\item a complete pair $(X,Y)$ and a number $y\in(0,1/4)$ such that \(\chi(X)\geq y^{24}\chi(G)\) and \(\chi(Y)\geq(1-y)\chi(G)\);
		\item a restrained blockade of length $k\geq2$ and mass at least $k^{-24}\chi(G)$.
	\end{enumerate}
\end{theorem}
\begin{proof}
	If $\chi(G)\leq2^{24}$, then $G$ has an edge because $\chi(G)\geq2$. The two endpoints of
	that edge form a complete, hence restrained, blockade of length $2$ and mass
	$1\geq2^{-24}\chi(G)$. We may therefore assume that $\chi(G)>2^{24}$. Since $2^7=128>101$, we have
	$2^{24}>101^4/16=c^{-4}$, and we may apply Lemma~\ref{lem:outer-decomp}.
	
	Its first outcome immediately yields outcome~\textup{(ii)} with $k=2$, and its second outcome is
	outcome~\textup{(i)}. It remains to treat its third outcome. Let
	$\epsilon\in[\chi(G)^{-1/4},c]$, and let $F$ be an $(\epsilon,\chi)$-dense induced subgraph with \(\chi(F)\geq\mu\epsilon^2\chi(G)\).
	Apply Lemma~\ref{lem:dense-blockade} to $F$. It produces a restrained blockade in $F$ of
	length $k$ and mass at least \(\alpha^2k^{-8}\chi(F)\), where \(k\geq\min\left\{\alpha\epsilon^{-1/2},\chi(F)^{5/12}\right\}\).
	We first show that the minimum equals its first entry. Since
	$\epsilon\geq\chi(G)^{-1/4}$, we have $\chi(G)\geq\epsilon^{-4}$, and hence \(\chi(F)\geq\mu\epsilon^{-2}\).
	Consequently,
	\[
	\chi(F)^{5/12}
	\geq\mu^{5/12}\epsilon^{-5/6}
	=\bigl(\mu^{5/12}\epsilon^{-1/3}\bigr)\epsilon^{-1/2}
	\geq\mu^{5/12}c^{-1/3}\epsilon^{-1/2}
	>\alpha\epsilon^{-1/2}
	\]
	by \eqref{ineq:const9}. Thus
	\begin{equation}\label{eq:k-lower-final}
		k\geq\alpha\epsilon^{-1/2}>\alpha c^{-1/2}>2.
	\end{equation}
	In particular, $k\geq3$.
	
	The mass of the blockade is at least \(\alpha^2k^{-8}\chi(F) \geq\mu\alpha^2\epsilon^2k^{-8}\chi(G)\).
	From \eqref{eq:k-lower-final}, $\epsilon\geq\alpha^2k^{-2}$, so \(\alpha^2k^{-8}\chi(F)
	\geq\mu\alpha^6k^{-12}\chi(G)\).
	Since $k\geq3$ and $\mu\alpha^6>3^{-12}$ by \eqref{ineq:const10}, it follows that \(\mu\alpha^6k^{-12}>k^{-24}\).
	Therefore the blockade has mass at least $k^{-24}\chi(G)$. A blockade restrained in $F$ is
	also restrained in $G$, because $F$ is induced and $\omega(F)\leq\omega(G)$. This proves
	outcome~\textup{(ii)}.
	This proves Theorem~\ref{thm:structure24}.
\end{proof}

We can now prove the main theorem.

\begin{proof}[\bf Proof of Theorem~\ref{thm:intro-main}]
	We proceed by induction on $t:=\omega(G)$. The assertion is immediate for $t\leq1$. Assume
	$t\geq2$ and that the result holds for all $\{P_5,C_5\}$-free graphs with clique number less
	than $t$. If $\chi(G)=1$, there is nothing to prove, so assume $\chi(G)\geq2$ and apply
	Theorem~\ref{thm:structure24}.
	Suppose first that outcome~\textup{(i)} holds. Since $X$ is complete to $Y$, we have \(\omega(X)+\omega(Y)\leq t\).
	If $\omega(X)\leq yt$, then $1\leq\omega(X)<t$ and the induction hypothesis gives \(\chi(G)
	\leq y^{-24}\chi(X)
	\leq y^{-24}\omega(X)^{24}
	\leq t^{24}\).
	If $\omega(X)>yt$, then $\omega(Y)< (1-y)t<t$, and \(\chi(G)
	\leq(1-y)^{-1}\chi(Y)
	\leq(1-y)^{-1}\omega(Y)^{24}
	\leq(1-y)^{23}t^{24}
	\leq t^{24}\).
	Suppose now that outcome~\textup{(ii)} holds, and let $(B_1,\ldots,B_k)$ be the restrained
	blockade. Since \(\sum_{i=1}^k\omega(B_i)\leq t\), there is an index $i$ with $\omega(B_i)\leq t/k<t$. The induction hypothesis yields \(\chi(G)
	\leq k^{24}\chi(B_i)
	\leq k^{24}\omega(B_i)^{24}
	\leq t^{24}\).
	This completes the induction.
\end{proof}

\section{Concluding remarks}\label{conremark}

We have shown that every $\{P_5,C_5\}$-free graph satisfies $\chi(G)\le \omega(G)^{24}$.
Building on the work of Nguyen~\cite{Nguyen2025}, the improvement from exponent~$40$ comes from two independent refinements.
The cutset argument gains one power of the density parameter because the $C_5$-free
condition reduces the covering multiplicity in Claim~\ref{clm:one-cover} from two to
one. The density increment then replaces the scale change $y\mapsto y^2$ by
$y\mapsto y^{6/5}$ and uses the full output of Lemma~\ref{lem:partition}.
Neither change alone yields exponent $24$.

We selected the constants in \eqref{eq:constants} to make the
three limiting branches compatible: the balanced complete-pair branch must exceed
$2^{-24}\chi(G)$; the dense branch must convert its $\epsilon^2\chi(G)$ mass into
$k^{-24}\chi(G)$; and the terminal block in Lemma~\ref{lem:density-step} must retain
more than a $\beta$ fraction of the original chromatic number.
The numerical margins are small in the balanced-pair branch (where we rely on
$\mu c^3>2^{-24}$), so a further reduction of the exponent within exactly the same parameter scheme would require new quantitative input.
We believe the exponent can be lowered, but our method is near its limit within this framework.
Our constants are chosen to close the induction; further optimization would require a new decomposition.
A more substantial improvement is likely to require a different structural mechanism rather than another round of numerical tuning.

We close with a few open problems:

\begin{problem}
	Is the exponent $24$ in Theorem~\ref{thm:intro-main} optimal for $\{P_5,C_5\}$-free graphs, or can a further improvement be obtained by a different method?
\end{problem}

More generally:

\begin{problem}
	What is the optimal $\chi$-binding function for the class of all $P_5$-free graphs?
\end{problem}

For longer paths, even the existence of a polynomial bound is widely open:

\begin{problem}
	For every integer $t\ge 6$, does there exist a constant $d$ such that every $P_t$-free graph $G$ satisfies $\chi(G) \le \omega(G)^d$?
\end{problem}

In particular, the following special case is also of interest:

\begin{problem}
	For every integer $t\ge 6$, does there exist a constant $d$ such that every $\{P_t,C_5\}$-free graph $G$ satisfies $\chi(G) \le \omega(G)^d$?
\end{problem}

%\subsection{Fixed constants and numerical checks}\label{subsec:constants}

%\begin{align}
%	\alpha c^{-1/2}&>2, \label{ineq:const1}\\
%	\alpha&<\frac{1}{2\sqrt2}-\frac{\sqrt c}{2}, \label{ineq:const2}\\
%	\beta-c&>\frac12, \label{ineq:const3}\\
%	\sqrt2+\frac{\sqrt c}{\beta}&<\gamma, \label{ineq:const4}\\
%	1-\gamma\alpha&>\beta, \label{ineq:const5}\\
%	c^{1/5}&<\beta, \label{ineq:const6}\\
%	\beta\alpha^{12/5}3^{18/5}&>1, \label{ineq:const7}\\
%	\mu c^3&>2^{-24}, \label{ineq:const8}\\
%	\mu^{5/12}c^{-1/3}&>\alpha, \label{ineq:const9}\\
%	\mu\alpha^6&>3^{-12}, \label{ineq:const10}\\
%	\mu&>c^{18}. \label{ineq:const11}
%\end{align}
%None of these comparisons relies on rounded decimal values. We include exact checks.

%Hence Problem~\ref{problem1} has an affirmative answer with the explicit constant $c=13$.

\vspace{6mm}

\n{\bf Acknowledgements:} %The authors would like to thank the
%anonymous referees for their constructive corrections and valuable
%comments on this paper, which have considerably improved the
%presentation of this paper.
%We thank Dr. Hongzhang Chen for his helpful discussions and for bringing Refs. \cite{Ajtai1980,Shearer1983} to our attention.
We thank Dr. Hongzhang Chen for helpful discussions concerning the choice of the numerical constants in \eqref{eq:constants}, in particular the construction of $\lambda$ and $\mu$.
This work was partially supported by the Foundation for Cultivated Young Talents of Fujian Province, China (Grant No. 2026350294), by the Natural Science Foundation of Fujian Province, China (Grant No. 2026J001968), and by the Fujian Key Laboratory of Granular Computing and Applications (Minnan Normal University), the Institute of Meteorological Big Data-Digital Fujian, and the Fujian Key Laboratory of Data Science and Statistics.

%\section*{Declarations}

%\subsection*{Author Contributions}
%
%K.-Y. Lan: conceptualization, methodology, formal analysis, writing – original draft, writing – review \& editing.
%W.-L. Zhong: investigation, validation, writing – review \& editing.
%Both authors contributed equally to this work and are listed in alphabetical order.

\subsection*{Data availability}
Data sharing is not applicable to this article as no datasets were generated or analysed during the current study.

\subsection*{Conflict of interest}
The authors declare that they have no known competing financial interests or personal relationships that could have appeared to influence the work reported in this paper.


\begin{thebibliography}{99}






\bibitem{BonomoChudnovskyMaceliSchaudtSteinZhong2018}
F. Bonomo, M. Chudnovsky, P. Maceli, O. Schaudt, M. Stein, and M. Zhong,
Three-coloring and list three-coloring of graphs without induced paths on seven vertices,
\textit{Combinatorica} \textbf{38} (2018), 779--801.

%\bibitem{BrandstadtKlembtMahfud2006}
%A. Brandst\"adt, T. Klembt, and S. Mahfud,
%\(P_6\)- and triangle-free graphs revisited: structure and bounded clique-width,
%\textit{Discrete Math. Theor. Comput. Sci.} \textbf{8} (2006), 173--188.

\bibitem{BrianskiDaviesWalczak2024}
M. Bria\'nski, J. Davies, and B. Walczak,
Separating polynomial \(\chi\)-boundedness from \(\chi\)-boundedness,
\textit{Combinatorica} \textbf{44} (2024), 1--8.



\bibitem{CameronHuangMerkel2021}
K. Cameron, S. Huang, and O. Merkel,
An optimal $\chi$-bound for $\{P_6, \text{diamond}\}$-free graphs,
\textit{J. Graph Theory} \textbf{97} (2021), 451--465.

\bibitem{CameronHuangPenevSivaraman2020}
K. Cameron, S. Huang, I. Penev, and V. Sivaraman,
The class of \(\{P_7,C_4,C_5\}\)-free graphs: decomposition, algorithms, and \(\chi\)-boundedness,
\textit{J. Graph Theory} \textbf{93} (2020), 503--552.

\bibitem{ChenWuXu2024}
R. Chen, D. Wu, and B. Xu,
Structure of some $\{P_7,C_4\}$-free graphs with application to colorings,
\textit{Discrete Appl. Math.} \textbf{357} (2024), 14--23.

\bibitem{ChenXu2025a}
R. Chen and B. Xu,
Structure and coloring of a family of $\{P_7,C_5\}$-free graphs,
\textit{Graphs Combin.} \textbf{41} (2025), 71.

%\bibitem{ChenXu2025b}
%R. Chen and B. Xu,
%Structure and coloring of $\{P_7, C_5, \text{diamond}\}$-free graphs,
%\textit{Discrete Appl. Math.} \textbf{372} (2025), 298--307.

\bibitem{ChoudumKarthickShalu2007}
S. A. Choudum, T. Karthick, and M. A. Shalu,
Perfectly coloring and linearly \(\chi\)-bound \(P_6\)-free graphs,
\textit{J. Graph Theory} \textbf{54} (2007), 293--306.

\bibitem{ChudnovskyHuangSpirklZhong2021}
M. Chudnovsky, S. Huang, S. Spirkl, and M. Zhong,
List 3-coloring graphs with no induced $P_6+rP_3$,
\textit{Algorithmica} \textbf{83} (2021), 216--251.

\bibitem{ChudnovskyKarthickMaceliMaffray2020}
M. Chudnovsky, T. Karthick, P. Maceli, and F. Maffray,
Coloring graphs with no induced five-vertex path or gem,
\textit{J. Graph Theory} \textbf{95} (2020), 527--542.

\bibitem{ChudnovskyMaceliStachoZhong2017}
M. Chudnovsky, P. Maceli, J. Stacho, and M. Zhong,
4-coloring $P_6$-free graphs with no induced 5-cycles,
\textit{J. Graph Theory} \textbf{84} (2017), 262--285.

\bibitem{ChudnovskySpirklZhong2024a}
M. Chudnovsky, S. Spirkl, and M. Zhong,
Four-coloring $P_6$-free graphs. I. Extending an excellent precoloring,
\textit{SIAM J. Comput.} \textbf{53} (2024), 111--145.

\bibitem{ChudnovskySpirklZhong2024b}
M. Chudnovsky, S. Spirkl, and M. Zhong,
Four-coloring $P_6$-free graphs. II. Finding an excellent precoloring,
\textit{SIAM J. Comput.} \textbf{53} (2024), 146--187.

\bibitem{ChudnovskyStacho2018}
M. Chudnovsky and J. Stacho,
3-colorable subclasses of $P_8$-free graphs,
\textit{SIAM J. Discrete Math.} \textbf{32} (2018), 1111--1138.

\bibitem{Erdos1959}
P. Erd\H{o}s,
Graph theory and probability,
\textit{Canad. J. Math.} \textbf{11} (1959), 34--38.

\bibitem{Esperet2017}
L. Esperet,
Graph colorings, flows and perfect matchings,
Habilitation Thesis, Universit\'e Grenoble Alpes, 2017.

\bibitem{EsperetLemoineMaffrayMorel2013}
L. Esperet, L. Lemoine, F. Maffray, and G. Morel,
The chromatic number of $\{P_5, K_4\}$-free graphs,
\textit{Discrete Math.} \textbf{313} (2013), 743--754.

\bibitem{GaspersHuang2019}
S. Gaspers and S. Huang,
Linearly $\chi$-bounding $\{P_6,C_4\}$-free graphs,
\textit{J. Graph Theory} \textbf{92} (2019), 322--342.

%\bibitem{Goedgebeur2023}
%J. Goedgebeur, S. Huang, Y. Ju, and O. Merkel,
%Colouring graphs with no induced six-vertex path or diamond,
%\textit{Theoret. Comput. Sci.} \textbf{941} (2023), 278--299.

\bibitem{GravierHoangMaffray2003}
S. Gravier, C. T. Ho\`ang, and F. Maffray,
Coloring the hypergraph of maximal cliques of a graph with no long path,
\textit{Discrete Math.} \textbf{272} (2003), 285--290.

\bibitem{Gyarfas1973}
A. Gy\'arf\'as,
On Ramsey covering numbers,
\textit{Colloq. Math. Soc. J\'anos Bolyai} \textbf{10} (1973), 801--816.

\bibitem{Gyarfas1987}
A. Gy\'arf\'as,
Problems from the world surrounding perfect graphs,
\textit{Zastos. Mat.} \textbf{XIX} (1987), 413--441.

%\bibitem{Hall1935}
%P.~Hall, On representatives of subsets,
%\textit{J. London Math. Soc.} \textbf{10} (1935), 26--30.

\bibitem{Huang2024}
S. Huang,
The optimal $\chi$-bound for $\{P_7,C_4,C_5\}$-free graphs,
\textit{Discrete Math.} \textbf{347} (2024), 114036.

\bibitem{HuangKarthick2021}
S. Huang and T. Karthick,
On graphs with no induced five-vertex path or paraglider,
\textit{J. Graph Theory} \textbf{97} (2021), 305--323.

%\bibitem{HuangZhouChang2026}
%S. Huang, Y. Zhou, and Y. Chang,
%The optimal chromatic bound for even-hole-free graphs without induced seven-vertex paths,
%\url{https://arxiv.org/pdf/2602.04403}.

\bibitem{JuJookenGoedgebeurHuang2026}
Y. Ju, J. Jooken, J. Goedgebeur, and S. Huang,
There are finitely many 5-vertex-critical $\{P_6,\text{bull}\}$-free graphs,
\textit{J. Graph Theory} \textbf{112} (2026), 255--266.

\bibitem{KarthickMaffray2019}
T. Karthick and F. Maffray,
Square-free graphs with no six-vertex induced path,
\textit{SIAM J. Discrete Math.} \textbf{33} (2019), 874--909.

\bibitem{Mycielski1955}
J. Mycielski,
Sur le coloriage des graphs,
\textit{Colloq. Math.} \textbf{3} (1955), 161--162.

\bibitem{Nguyen2025}
T. H. Nguyen,
Polynomial $\chi$-boundedness for excluding $P_5$,
\url{https://arxiv.org/pdf/2512.24907}.





%\bibitem{Randerath1998}
%B. Randerath,
%The Vizing Bound for the Chromatic Number Based on Forbidden Pairs,
%Ph.D. thesis, Shaker Verlag, Aachen, 1998.
%
%\bibitem{Schiermeyer2020}
%I. Schiermeyer,
%Polynomial $\chi$-binding functions for $P_5$-free graphs, available at \url{https://www.iti.zcu.cz/colourings20/talks/Schiermeyer.pdf}.

\bibitem{SchiermeyerRanderath2019}
I. Schiermeyer and B. Randerath,
Polynomial $\chi$-binding functions and forbidden induced subgraphs: a survey,
\textit{Graphs Combin.} \textbf{35} (2019), 1--31.



\bibitem{ScottSeymourSpirkl2023}
A. Scott, P. Seymour, and S. Spirkl,
Polynomial bounds for chromatic number. IV: A near-polynomial bound for excluding the five-vertex path,
\textit{Combinatorica} \textbf{43} (2023), 845--852.

%\bibitem{Seinsche1974}
%D. Seinsche, On a property of the class of $n$-colorable graphs,
%\textit{J. Combin. Theory Ser. B} \textbf{16} (1974), 191--193.















%\bibitem{BrauseGeisser2021}
%C. Brause and M. Gei{\ss}er,
%On graphs without an induced path on 5 vertices and without an induced dart or kite,
%in: \textit{Extended Abstracts EuroComb 2021}, Trends Math. \textbf{14} (2021), 311--317.







%\bibitem{ChudnovskySeymourRobertsonThomas2010}
%M. Chudnovsky, P. Seymour, N. Robertson, and R. Thomas,
%$K_4$-free graphs with no odd holes,
%\textit{J. Combin. Theory Ser. B} \textbf{100} (2010), 313--331.



%\bibitem{HongXu2025}
%X. Hong and B. Xu,
%Coloring of $\{P_6, \text{dart}, K_4\}$-free graphs,
%\textit{Discrete Appl. Math.} \textbf{365} (2025), 223--230.



%\bibitem{KarthickMaffray2016}
%T. Karthick and F. Maffray,
%Vizing bound for the chromatic number on some graph classes,
%\textit{Graphs Combin.} \textbf{32} (2016), 1447--1460.



%\bibitem{KarthickMishra2018}
%T. Karthick and S. Mishra,
%On the chromatic number of $\{P_6, \text{diamond}\}$-free graphs,
%\textit{Graphs Combin.} \textbf{34} (2018), 677--692.





%\bibitem{Lovasz1972}
%L. Lov\'asz,
%A characterization of perfect graphs,
%\textit{J. Combin. Theory Ser. B} \textbf{13} (1972), 95--98.





\bibitem{Sumner1981}
D. P. Sumner,
Subtrees of a graph and chromatic number,
in \textit{The Theory and Applications of Graphs},
G. Chartrand, ed.,
John Wiley \& Sons, New York (1981), 557--576.

%\bibitem{WoegingerSgall2001}
%G. J. Woeginger and J. Sgall,
%The complexity of coloring graphs without long induced paths,
%\textit{Acta Cybernet.} \textbf{15} (2001), 107--117.


\bibitem{Zykov1949}
A. A. Zykov,
On some properties of linear complexes,
\textit{Mat. Sb. N.S.} \textbf{24} (1949), 163--188.

%\bibitem{KarthickMaffray2016}
%T. Karthick and F. Maffray,
%Vizing bound for the chromatic number on some graph classes,
%\textit{Graphs Combin.} \textbf{32} (2016), 1447--1460.





























%\bibitem{chudnovsky2006}
%M. Chudnovsky, N. Robertson, P. Seymour, and R. Thomas,
%The strong perfect graph theorem,
%\textit{Ann. of Math. (2)} \textbf{164} (2006), 51--229.





%\bibitem{GaspersHuang2019}
%S. Gaspers and S. Huang,
%$(2P_2,K_4)$-free graphs are 4-colorable,
%\textit{SIAM J. Discrete Math.} \textbf{33} (2019), 1095--1120.




%\bibitem{Gyarfas2023}
%A. Gy\'arf\'as,
%Problems close to my heart. With a preface by G. S\'ark\H{o}zy,
%\textit{European J. Combin.} \textbf{111} (2023), 103695.



%\bibitem{lovasz1972characterization}
%L. Lov\'asz,
%A characterization of perfect graphs,
%\textit{J. Combin. Theory Ser. B} \textbf{13} (1972), 95--98.



%\bibitem{scott2020induced}
%A. Scott and P. Seymour,
%Induced subgraphs of graphs with large chromatic number. VII. Gy\'arf\'as's complementation conjecture,
%\textit{J. Combin. Theory Ser. B} \textbf{142} (2020), 43--55.







%\bibitem{Wagon1980}
%S. Wagon,
%A bound on the chromatic number of graphs without certain induced subgraphs,
%\textit{J. Combin. Theory Ser. B} \textbf{29} (1980), 345--346.







	

%\bibitem{CCDO2026}
%M. Chudnovsky, L. Cook, J. Davies and S. Oum,
%Reuniting \(\chi\)-boundedness with polynomial \(\chi\)-boundedness,
%\textit{J. Combin. Theory Ser. B} \textbf{176} (2026), 30--73.








      
\end{thebibliography}
\end{document}